\documentclass{gtpart}

\usepackage{amssymb,amsmath,amscd,amstext,amsthm}
\usepackage[all,cmtip]{xy}

\newtheorem{theorem}{Theorem}[section]

\newtheorem{lemma}[theorem]{Lemma}

\begin{document}

\title[Explicit generators for relation module]{Explicit generators for the relation module in the example of Gruenberg--Linnell}

\author[W. H. Mannan]{\tt W. H. Mannan\\
School of Computing, Engineering and Mathematics, 
University of Brighton, Cockroft Building, Lewes Road, Brighton BN\textup{2 4}GJ\\email: wajid@mannan.info}

\begin{abstract}
 Gruenberg and Linnell showed that the standard relation module of a free product of $n$ groups of the form $C_r \times \mathbb{Z}$ could be generated by just $n+1$ generators, raising the possibility of a relation gap.  We explicitly give such a set of generators.

 

\end{abstract}

\maketitle

\section{Introduction}
The relation gap problem for a finite set of generators of a finitely presented group asks if the number of relations needed to present the group, exceeds the number needed to generate the relation module (in which case the group is said to have a relation gap)\cite{Harl}.  This problem is closely related to major questions in low dimensional topology, most directly Wall's D(2) problem \cite{John1}, as a group with relation gap would likely lead to a cohomologically 2 dimensional finite cell complex, not homotopy equivalent to a finite 2-complex \cite{Mann,Mann1}.

 A number of potential groups with relation gap have been posited, based on the idea of taking free products of groups with some parameter mutually coprime.  The idea here is that the Chinese remainder theorem allows one to pick a set of generators for the relation module far smaller than any known set of relators which present the group.  Gruenberg and Linnell provided an example \cite{Grue} taking free products of perhaps the simplest such groups: $C_r \times \mathbb{Z}$.   
 
Whilst it is known that $C_r \times \mathbb{Z}$ satisfies the D(2) property  \cite{Time}, it follows from \cite{Grue} that the free product of such groups may not (and thus solve Wall's D(2) problem).
 
 However unlike other such examples (e.g. \cite{Brid})  the generators for the relation module are not given in any obvious way in \cite{Grue} and the arguments used to prove their existence are highly technical.   This has hindered the study of these examples in various ways.  For example, without explicit generators one cannot explicitly construct potential counterexamples to Wall's D(2) problem (as done in \cite{Brid}).  On the other hand one cannot seek to modify the generators of the relation module to obtain a presentation of the group, as has been done in another case \cite[Satz 2.35]{Hoga}.
 
 We rectify this by providing explicit generators for the relation module, with the simplest form one could hope for.  We thus dispel the notion that the Gruenberg--Linnell example is more obscure than later examples.  
 
 As an application, in \S\ref{D2} we use our generators to give an explicit description of a candidate for the solution of Wall's D2 problem. 
 
 We note that the author of \cite{Harl} pointed out that the late Karl Gruenberg mentioned similar constructions to ours in private correspondence (now lost).

\section{The relation module}
\noindent Given pairwise coprime integers $r_1,\cdots, r_n\geq 2$, let
 \begin{align}\qquad G&=&&\substack{n\\ * \\{i=1}} (C_{r_i} \times \mathbb{Z})\nonumber \\&=&& \langle a_i,\,b_i,\, i=1,\cdots,n\vert\,\, R_i,\,S_i,\,i=1,\cdots,n\rangle,\label{pres}\end{align}
where the $R_i,S_i$ are the following elements of the free group $F$ on the $a_i,b_i$:\[R_i=[a_i,b_i]=a_ib_i{a_i}^{-1}{b_i}^{-1},\qquad\qquad S_i={a_i}^{r_i}.\]
For $i=1,\cdots,n$ let $\Sigma_i, \Gamma_i$ denote the following elements of the group ring $\mathbb{Z}[G]$:
\[
\Sigma_i=\substack{{r_i-1}\\ \sum\\{j=0}} {a_i}^j,\qquad\qquad \Gamma_i=\substack{{r_i-1}\\ \sum\\{j=0}} j{a_i}^j.
\]
Thus for each $i$ we have:\begin{align}
(1-a_i)\Sigma_i=0,\qquad {\Sigma_i}^2=\Sigma_i r_i,\qquad (1-a_i)\Gamma_i=\Sigma_i-r_i.\label{gprngid}\end{align}

The relation module $M$ associated to (\ref{pres}) is the abelianization of the normal closure of the $R_i,S_i\in F$.  Regarding this as an additive group, we have a well defined (right) action of $G$ on $M$ given by conjugation in $F$.  Let the $D_i,E_i \in M$ denote the images of the $R_i,S_i\in F$ respectively.  Thus the $D_i,E_i$ generate $M$ as a right $\mathbb{Z}[G]$ module.

For each $i$ we have the following identities in $F$:
\begin{align*}({a_i}^{-1}R_i{a_i})({a_i}^{-2}R_i{a_i}^2)\cdots({a_i}^{-r_i}R_i{a_i}^{r_i})&&
{a_i}^{-1}{S_i}^{-1}a_iS_i=e.
\\=({a_i}^{-1}R_i)^{r_i}{a_i}^{r_i}=
b_i{a_i}^{-r_i}{b_i}^{-1}{a_i}^{r_i}\\=
b_i{S_i}^{-1}{b_i}^{-1}S_i,\end{align*}

Thus in $M$ we have:
\begin{align}D_i\Sigma_i=E_i(1-{b_i}^{-1}),\qquad\qquad\qquad E_i(1-a_i)=0.\label{relid}\end{align}

\begin{lemma}
For each $i$ we have \[(E_i\,\,+\,\,D_i(1-a_i))\quad ((1-{b_i}^{-1})\Sigma_i\,\,+\,\,(\Sigma_i-r_i)\Gamma_i)\qquad=\qquad D_i {r_i}^2.\] \label{calc}
\end{lemma}

\begin{proof}
Multiplying out the LHS using (\ref{gprngid}), (\ref{relid}), we get the following four terms:
\begin{align*}
E_i(1-{b_i}^{-1})\Sigma_i&=&&D_i{\Sigma_i}^2&=&D_i\Sigma_ir_i,\\
E_i(\Sigma_i-r_i)\Gamma_i&=&&E_i(1-a_i){\Gamma_i}^2&=&0,\\
D_i(1-a_i)(1-{b_i}^{-1})\Sigma_i&=&&D_i(1-{b_i}^{-1})(1-a_i)\Sigma_i&=&0,\\
D_i(1-a_i)(\Sigma_i-r_i)\Gamma_i&=&&D_i(\Sigma_i-r_i)^2&=&D_i{r_i}^2-D_i\Sigma_ir_i.
\end{align*}
Adding the terms gives $D_i {r_i}^2$, as required. 
\end{proof}

\begin{theorem}\label{main}
Let \[X_i=E_i\,+\,D_i(1-a_i),\quad\,\,i=1,\cdots,n \qquad{\rm and}\qquad X_{n+1}=\substack{n\\ \sum\\{i=1}} D_i.\]  Then the $X_i$, $i=1,\cdots,n+1$ generate the whole of $M$.
\end{theorem}

\begin{proof}
From Lemma \ref{calc} we know that each $D_i{r_i}^2$ lies in the span of the $X_i$.  As the $r_i$ are pairwise coprime, the Chinese remainder theorem tells us that the $D_i$ are in the span of the $D_i {r_i}^2$ and $X_{n+1}$.  Finally note that each $E_i=X_i-D_i(1-a_i)$.  Thus all the $D_i, E_i$ lie in the span of the $X_i$ as required.
\end{proof}

\section{A cohomologically 2 dimensional 3--complex} \label{D2}

 The most direct consequence of Theorem \ref{main} is that those seeking to find relators in the $F$ to realize the $X_i$ (and thus disprove that this  example resolves the relation gap problem), will now know what they are seeking to realize, rather than just its existence.  
  
One approach to finding such relators is to apply elementary transformations to (\ref{pres}) to obtain a presentation containing several trivial relators (empty products of the $a_i,b_i$).  Such trivial relators yield basis elements for a  free factor in the module $C_2(Y)$ defined below.  Thus another application of Theorem \ref{main} is Lemma \ref{factor} in this section, which explicitly gives  a basis for such a factor.

 However our main purpose in deriving Lemma \ref{factor}  is for another application: explicitly presenting a potential solution to Wall's D2 problem.  This problem asks if every finite cohomologically 2 dimensional 3--complex is homotopy equivalent to a finite 2--complex (all spaces considered are finite CW complexes).  In every other dimension it is known that cohomological dimension reflects the minimal dimension in the homotopy type.

From our explicit generating set of the relation module one can construct a cohomologically 2 dimensional 3 complex that has the potential of being a counterexample to the D(2) conjecture. One  construction is given in the proof of Dyer's theorem (see \cite{Harl}) and was used in Bridson/Tweedale \cite{Brid}. We follow a more direct route.

Starting with $Y$, the Cayley complex of (\ref{pres}), we will attach a finite number of 3--cells to obtain a finite cohomologically 2 dimensional 3--complex $Y'$, which is not known to be homotopy equivalent to any finite 2--complex.  Indeed there is no known finite 2--complex with the same fundamental group and Euler characteristic as $Y'$.

The cellular chain complex of $\tilde{Y}$ (the universal cover of $Y$) is an algebraic complex of modules and linear maps over  $\mathbb{Z}[G]$:

\vspace{.4mm}

\xymatrix{C_*(\tilde{Y})\colon=&0\ar@{.>}[r]&\pi_2(Y)\ar@{.>}[r]&C_2(\tilde{Y})\ar[r]^{d_2}&C_1(\tilde{Y})\ar[r]^{d_1}&C_0(\tilde{Y})\ar@{.>}[r]&\mathbb{Z}\ar@{.>}[r] &0}

\vspace{1.6mm}

\noindent This sequence is exact, as the kernel of $d_2$ is identified with $\pi_2(\tilde{Y})=\pi_2(Y)$ via the Hurewicz isomorphism.  The module $C_2(\tilde{Y})$ is freely generated by elements $\hat{D_i},\hat{E_i}$ corresponding to the relators $R_i,S_i$ respectively.

We may factorize $d_2$ through $M$:

\bigskip
\xymatrix{&&&&&C_2(\tilde{Y})\ar[dr]^f\ar[rr]^{d_2}&&C_1(\tilde{Y})\\
&&&&&&M\ar[ur]_\iota}

\medskip
\noindent where $\iota\colon M \to C_1(\tilde{Y})$ is an inclusion and $f\colon C_2(\tilde{Y}) \to M$ is given by:\[f\colon\,\, \hat{D_i}\mapsto D_i,\quad  \hat{E_i} \mapsto E_i,\] for $i=1,\cdots,n$.  Thus to attach 3-cells to $Y$ one must specify how their boundaries are attached to $Y$ via elements of $\pi_2(Y)={\rm ker}\, d_2= {\rm ker} f$.  

Let

\vspace{-9mm}

\[\hat{X_i}=\hat{E_i}\,+\,\hat{D_i}(1-a_i),\quad\,\,i=1,\cdots,n \qquad{\rm and}\qquad \hat{X}_{n+1}=\substack{n\\ \sum\\{i=1}}\hat{ D_i}.\]  

\vspace{-2mm}

By Theorem \ref{main} we may pick $\alpha_1,\cdots,\alpha_{n-1} \in {\rm ker} f$ such that $\alpha_i$ differs from $\hat{D_i}$ by a linear combination of the $\hat{X_i}$.  Now construct $Y'$ from $Y$ by attaching a 3-cell $B_i$ with attaching map given by $\alpha_i$, for each $i=1,\cdots,n-1$.

\begin{lemma}\label{factor}
We have a basis for $C_2(\tilde{Y})$ given by $\alpha_1,\cdots,\alpha_{n-1}, \hat{X_1},\cdots,\hat{X}_{n+1}$.
\end{lemma}

\vspace{-4mm}
\begin{proof}
Starting with the set  $\alpha_1,\cdots,\alpha_{n-1}, \hat{X_1},\cdots,\hat{X}_{n+1}$ and replacing each $\alpha_j$ with itself minus  appropriate multiples of the $\hat{X_i}$, we are left with the set $\hat{D_1},\cdots,\hat{D}_{n-1}, \hat{X_1},\cdots,\hat{X}_{n+1}$.

Then replace $\hat{X}_{n+1}$ with itself minus the sum of the $\hat{D_1},\cdots, \hat{D}_{n-1}$, to get the set $\hat{D_1}\cdots,\hat{D_{n}}, \hat{X_1},\cdots,\hat{X}_{n}$. 

Finally for $i=1,\cdots,n$ replace each $\hat{X_i}$ with itself minus an appropriate multiple of $\hat{D_i}$, to obtain  $\hat{D_1}\cdots,\hat{D_{n}}, \hat{E_1},\cdots,\hat{E}_{n}$, which is a basis of $C_2(\tilde{Y})$ .

As the property of being a basis is invariant under the operations we performed, $\alpha_1,\cdots,\alpha_{n-1}, \hat{X_1},\cdots,\hat{X}_{n+1}$ is also a basis of $C_2(\tilde{Y})$.
\end{proof}

Thus we may write $C_2(\tilde{Y})\cong A\oplus X$ where $A$ is generated by the $\alpha_i$ and $X$ is generated by the $\hat{X_i}$.  Consider the boundary map \[d'_3\colon C_3(\tilde{Y'}) \to C_2(\tilde{Y'})\cong C_2(\tilde{Y})\cong A \oplus X\]
The generator in $C_3(\tilde{Y'})$ corresponding to $B_i$ maps to $\alpha_i \in A$.  Thus $d'_3$ identifies $C_3(\tilde{Y'})$ with the summand $A$ in $C_2(\tilde{Y'})$ and we know that $Y'$ is cohomologically 2 dimensional.  The Euler characteristic of $Y'$ is $1-2n+2n-(n-1)=2-n$.  We conclude:

 \begin{theorem}
 The finite 3-complex $Y'$ is cohomologically 2 dimensional, has fundamental group $G$ and Euler characteristic $2-n$.
  \end{theorem} 

If $Y'$ were homotopy equivalent to a finite 2-complex $Z$, then on contracting $Z$ by a maximal tree in its 1-skeleton one would obtain a Cayley complex $Z'$.  The Euler characteristic of $Z'$ would be $2-n$ so the underlying presentation of $G$ would have $n-1$ more generators than relators.  However no known presentations of $G$ have fewer relators than generators.  Thus $Y'$ is a plausible candidate for a solution to Wall's D2 problem.

\vspace{-1mm}

{\bf Acknowledgments}
The author wishes to thank Jens Harlander for reading a preprint of this work and providing historical context.

\end{document}